\documentclass[oneside]{amsart}
\usepackage[T1]{fontenc}
\usepackage[latin1]{inputenc}
\usepackage{amssymb}

\makeatletter
 \theoremstyle{plain}    
 \newtheorem{thm}{Theorem}[section]
 \numberwithin{equation}{section} 
 \numberwithin{figure}{section} 
 \theoremstyle{plain}
 \theoremstyle{plain}    
 \newtheorem{cor}[thm]{Corollary} 
 \theoremstyle{plain}    
 \newtheorem{lem}[thm]{Lemma} 
 \theoremstyle{plain}    
 \newtheorem{prop}[thm]{Proposition} 

\makeatother
\begin{document}

\title{The Horn Conjecture for Sums of Compact Selfadjoint Operators}

\author{H. Bercovici, W.S. Li and D. Timotin}

\thanks{The first author was supported in part by grants from the National
Science Foundation. The second and third authors thank Indiana University for its hospitality while this paper was written.}

\begin{abstract}
We determine the possible eigenvalues of compact selfadjoint operators
$A,B^{(1)},B^{(2)},\dots,B^{(m)}$ with the property that $A=B^{(1)}+B^{(2)}+\cdots+B^{(m)}$.
When all these operators are positive, the eigenvalues were known
to be subject to certain inequalities which extend Horn's inequalities
from the finite-dimensional case when $m=2$. We find the proper extension
of the Horn inequalities and show that they, along with their reverse
analogues, provide a complete characterization. Our results also allow
us to discuss the more general situation where only some of the eigenvalues
of the operators $A$ and $B^{(k)}$ are specified. A special case
is the requirement that $B^{(1)}+B^{(2)}+\cdots+B^{(k)}$ be positive
of rank at most $\rho\ge1$.
\end{abstract}
\maketitle

\section{Introduction}

Given an $N\times N$ complex Hermitian matrix $A$, we denote by
$\Lambda(A)$ the sequence\[
\lambda_{1}(A)\ge\lambda_{2}(A)\ge\cdots\ge\lambda_{N}(A)\]
of its eigenvalues, in decreasing order and listed according to their
multiplicities. A. Horn \cite{horn} conjectured a characterization
of the set of triples $(\alpha,\beta,\gamma)\in(\mathbb{R}^{N})^{3}$
with the property that there exist Hermitian matrices $A,B,C$ such
that $\Lambda(A)=\alpha,\Lambda(B)=\beta$, $\Lambda(C)=\gamma$,
and $A=B+C$. More precisely, for every integer $r\in\{1,2,\dots,N-1\}$
he introduced a collection $T_{r}^{N}$ of triples $(I,J,K)$ of subsets
of $\{1,2,\dots,N\}$ such that $|I|=|J|=|K|=r$; here we use $|I|$
to denote the cardinality of $I$. We will call \emph{Horn triples}
the elements $(I,J,K)$ of $T_{r}^{N}$. Horn conjectured that the
triples $(\Lambda(A),\Lambda(B),\Lambda(C))$ (with $A=B+C$) are
precisely those triples $(\alpha,\beta,\gamma)$ of decreasing sequences
satisfying the \emph{trace identity} \[
\sum_{i=1}^{N}\alpha_{i}=\sum_{j=1}^{N}\beta_{j}+\sum_{k=1}^{N}\gamma_{k},\]
and the \emph{Horn inequalities} \begin{equation}
\sum_{i\in I}\alpha_{i}\le\sum_{j\in J}\beta_{j}+\sum_{k\in K}\gamma_{k}\label{ineq:Horn}\end{equation}
for all Horn triples $(I,J,K)\in T_{r}^{N}$ with $r<N$. The study
and eventual proof of this conjecture are chronicled by W. Fulton
in his excellent survey \cite{Ful-BAMS}.

Assume now that $A$ is a positive compact operator on a complex Hilbert
space. In this case we denote by \[
\Lambda_{+}(A)=\{\lambda_{1}(A)\ge\lambda_{2}(A)\ge\cdots\}\]
the sequence of its eigenvalues, in decreasing order and repeated
according to their multiplicities. Note that $\lambda_{j}(A)>0$ for
all $j$ when $A$ has infinite rank, so that the possible eigenvalue
$0$ is not listed in this case. The problem of characterizing the
triples $(\Lambda_{+}(B+C),\Lambda_{+}(B),\Lambda_{+}(C))$ when $B$
and $C$ are positive compact operators was studied by S. Friedland
\cite{key-230} and Fulton \cite{Ful-LAA}. They observed that, if
$(\alpha,\beta,\gamma)=(\Lambda_{+}(B+C),\Lambda_{+}(B),\Lambda_{+}(C))$,
the Horn inequalities (\ref{ineq:Horn}) must be satisfied for all
$(I,J,K)\in T_{r}^{N}$, for all $N\ge1$ and all $r<N$. Conversely,
if $(\alpha,\beta,\gamma)$ is a triple of decreasing sequences tending
to zero, and they satisfy all the Horn inequalities, then there exist
positive compact operators $A,B,C$ such that $\Lambda_{+}(A)=\alpha$,
$\Lambda_{+}(B)=\beta$, $\Lambda_{+}(C)=\gamma$, and $A\le B+C$.
If in addition\[
\sum_{i=1}^{\infty}\alpha_{i}=\sum_{j=1}^{\infty}\beta_{j}+\sum_{k=1}^{\infty}\gamma_{k}<\infty,\]
we must have $A=B+C$. If we only know that $\sum_{j=1}^{\infty}\beta_{j}<\infty$,
replacing this trace identity by the conditions $\gamma\le\alpha$
and\[
\sum_{j=1}^{\infty}\beta_{j}=\sum_{k=1}^{\infty}(\alpha_{k}-\gamma_{k})\]
leads to analogous results; this is shown in \cite{BLS}. When all
of these sums are infinite the trace identities are useless. The results
of \cite{BLS} do provide a characterization even in this case, based
on an extension of the Littlewood-Richardson rule.

We will show that the trace identity can be replaced in general by
a system of inequalities. To describe these inequalities we need some
notation. Given a finite set $I\subset\mathbb{N}=\{1,2,\dots\}$,
we denote by $I^{c}=\mathbb{N}\setminus I$ the complement of $I$,
and by $I_{p}^{c}$ the set consisting of the $p$ smallest elements
of $I^{c}$. The following statement is a particular case of Theorem
\ref{thm:positive-horn}. It will be convenient to consider that $(\varnothing,\varnothing,\varnothing)$
is a Horn triple.

\begin{thm}
\label{thm:a=3Db+c}Let $\alpha,\beta,\gamma$ be three decreasing
sequences with limit zero. The following conditions are equivalent:
\begin{enumerate}
\item There exist positive compact operators $B$ and $C$ such that $\Lambda_{+}(B+C)=\alpha$,
$\Lambda_{+}(B)=\beta$ and $\lambda_{+}(C)=\gamma$. 
\item For every Horn triple $(I,J,K)$, and for all positive integers $p,q$
we have the Horn inequality\[
\sum_{i\in I}\alpha_{i}\le\sum_{j\in J}\beta_{j}+\sum_{k\in K}\gamma_{k},\]
and the extended reverse Horn inequality\[
\sum_{i\in I_{p+q}^{c}}\alpha_{i}\ge\sum_{j\in J_{p}^{c}}\beta_{j}+\sum_{k\in K_{q}^{c}}\gamma_{k}.\]

\end{enumerate}
\end{thm}
The extended reverse Horn inequalities may at first seem to be altogether
wrong, since the index sets are of different cardinalities. It is
important to remember though that the problem we consider is not invariant
under addition of constant multiples of the identity operator to $B$
and $C$, which is the case in finite dimensions. These cardinalities
are in fact equalized in the more general situation of compact selfadjoint
operators. We will consider more than two summands, as is also done
in \cite{Ful-LAA,key-230}. Thus, we produce a characterization of
the set of all $(m+1)$-tuples $(\Lambda_{0}(A),\Lambda_{0}(B^{(1)}),\Lambda_{0}(B^{(2)}),\dots,\Lambda_{0}(B^{(m)}))$,
where the $B^{(j)}$ are compact selfadjoint operators and $A=\sum_{j=1}^{m}B^{(j)}$.
The eigenvalue sequence $\Lambda_{0}(A)$ will be described in Section
4 when $A$ is a compact selfadjoint operator. This question can naturally
be put in a more symmetric form by setting $A^{(0)}=-B$, so that
$\sum_{j=0}^{m}B^{(j)}=0$. The finite-dimensional version of these
problems is discussed in Fulton \cite{Ful-BAMS}, and it will be described
in Section 2 below. A more general problem was considered by Buch
\cite{buch}, and we will also extend his results to the compact selfadjoint
case. In fact, we will find necessary and sufficient conditions for
the existence of compact selfadjoint operators $B^{(k)}$ and $A=\sum_{k=1}^{m}B^{(k)}$
when their eigenvalues are only partially specified. This is a new
result even in the matrix case, and it answers a question posed in
\cite{Ful-BAMS}.

The remainder of the paper is organized as follows. In Section 2 we
describe in detail the Horn conjecture in finite dimensions, and we
deduce some important properties of Horn triples. In Section 3 we
prove a finite-dimensional interpolation result whose infinite-dimensional
analogue yields the Horn conjecture in Section 4. Section 5 contains
the discussion of matrices and operators with partially specified
eigenvalues. The case of positive operators is discussed in Section
6, where the extension to $m\ge2$ of Theorem \ref{thm:a=3Db+c} is
proved. We also dicuss briefly the hive formulation of the Littlewood-Richardson
rule.

\section{Horn Inequalities in Finite Dimensions}

It will occasionally be convenient to view a set $I=\{ i_{1}<i_{2}<\cdots<i_{r}\}$
of natural numbers as a function $I:\{1,2,\dots,r\}\to\mathbb{N}$,
i.e., $I(\ell)=i_{\ell}$. Thus, if $I'\subset\{1,2,\dots,r\}$, there
is a well-defined subset $I\circ I'\subset I$. We will also set $[n]=\{1,2,\dots,n\}$
for $n\in\mathbb{N}$.

Fix an integer $m\ge2$. Given integers $N,r$ such that $N\ge0$
and $0\le r\le N$, we define a collection $T_{r}^{N}(m+1)$ of $(m+1)$-tuples
$(I,J^{(1)},J^{(2)},\dots,J^{(m)})$ of subsets of $[N]$ such that
$|I|=|J^{(1)}|=\cdots=|J^{(m)}|=r$, $0\le r\le N$; when $m=2$,
these will be precisely the Horn triples. We proceed by induction
on $N$. When $N=0$, there is only one set to define: $T_{0}^{0}(m+1)$
consists of the $(m+1)$-tuple $(\varnothing,\varnothing,\dots,\varnothing)$.
Assume that the sets $T_{s}^{M}$ have been defined for all $M<N$
and $0\le s\le M$. The $(m+1)$-tuples $(I,J^{(1)},J^{(2)},\dots,J^{(m)})$
in $T_{r}^{N}(m+1)$ are then subject to the requirements\begin{equation}
\sum_{\ell=1}^{r}(I(\ell)-\ell)=\sum_{k=1}^{m}\sum_{\ell=1}^{r}(J^{(k)}(\ell)-\ell),\label{equation-horn-m-tuples}\end{equation}
and\begin{equation}
\sum_{\ell=1}^{s}(I\circ I'(\ell)-\ell)\ge\sum_{k=1}^{m}\sum_{\ell=1}^{s}(J^{(k)}\circ J'^{(k)}(\ell)-\ell)\label{equ-m-horn-ineq}\end{equation}
for every $s<r$ and every $(I',J'^{(1)},\dots,J'^{(m)})\in T_{s}^{r}(m+1)$.
Observe that the set $T_{N}^{N}(m+1)$ consists of the $(m+1)$-tuple
defined by $I=J^{(1)}=\cdots=J^{(m)}=[N]$. We will also consider
the larger set $\overline{T}_{r}^{N}(m+1)$ of $(m+1)$-tuples for
which the identity (\ref{equation-horn-m-tuples}) is replaced by\[
\sum_{\ell=1}^{r}(I(\ell)-\ell)\ge\sum_{k=1}^{m}\sum_{\ell=1}^{r}(J^{(k)}(\ell)-\ell).\]

The following result (in an equivalent form) is proved in \cite{Ful-LAA}.
For a set $I\subset[N]$ we denote $I_{\text{sym}}=\{ N+1-i:i\in I\}$.

\begin{thm}
\label{Theo:Horn-conj-m}Let $\alpha,\beta^{(1)},\beta^{(2)}\dots,\beta^{(m)}\in\mathbb{R}^{N}$
be decreasing vectors. The following conditions are equivalent:
\begin{enumerate}
\item There exist Hermitian $N\times N$-matrices $A,B^{(1)},B^{(2)},\dots,B^{(m)}$
such that $\Lambda(A)=\alpha$, $\Lambda(B^{(k)})=\beta^{(k)}$ for
$k=1,2,\dots,m$, and $A\le B^{(1)}+B^{(2)}+\cdots+B^{(m)}$.
\item For every $r\le N$ and every $(I,J^{(1)},J^{(2)},\dots,J^{(m)})\in T_{r}^{N}(m+1)$
we have the Horn inequality\[
\sum_{\ell=1}^{r}\alpha_{I(\ell)}\le\sum_{k=1}^{m}\sum_{\ell=1}^{r}\beta_{J^{(k)}(\ell)}.\]

\item For every $r\le N$ and every $(I,J^{(1)},J^{(2)},\dots,J^{(m)})\in T_{r}^{N}(m+1)$
we have the inequality\[
\sum_{i\notin I_{\text{{\rm sym}}}}\alpha_{i}\le\sum_{k=1}^{m}\sum_{j\notin I_{{\rm sym}}^{(k)}}\beta_{j}.\]

\end{enumerate}
\end{thm}
Note that the requirement (2) (for $r=N$) or (3) (for $r=0$) yields
\[
\sum_{\ell=1}^{N}\alpha_{\ell}\le\sum_{k=1}^{m}\sum_{\ell=1}^{N}\beta_{\ell}.\]
 Replacing this requirement by the trace condition\[
\sum_{\ell=1}^{N}\alpha_{\ell}=\sum_{k=1}^{m}\sum_{\ell=1}^{N}\beta_{\ell}\]
implies the equality $A=B^{(1)}+\cdots+B^{(m)}$. It can also be shown
that conditions (2) and (3) are not merely equivalent: they are precisely
the same. We will not need this stronger assertion, so we do not include
a proof. The interested reader will be able to supply an inductive
argument based on the definion of the sets $T_{r}^{N}$.

An immediate consequence of this result is an alternate characterization
of the elements of $T_{r}^{N}(m+1)$. We can associate with each subset
$I\subset\mathbb{N}$ of cardinality $r$ an integer partition $\pi(I)$
of length $r$ as follows:

\[
\pi(I)=\{ I(r)-r\ge I(r-1)-(r-1)\ge\cdots\ge I(1)-1\}.\]

\begin{cor}
\label{cor:Horn-indices-are-eigenvalues}Let $I,J^{(1)},\dots,J^{(m)}\subset\{1,2,\dots,N\}$
be such that $|I|=|J^{(1)}|=\cdots=|J^{(m)}|=r$. The following conditions
are equivalent:
\begin{enumerate}
\item $(I,J^{(1)},\dots,J^{(m)})\in T_{r}^{N}(m+1)$ $($respectively, $(I,J^{(1)},\dots,J^{(m)})\in\overline{T}_{r}^{N}(m+1)$$)$.
\item There exist Hermitian $r\times r$-matrices $A,B^{(1)},B^{(2)},\dots,B^{(m)}$
such that $\Lambda(A)=\pi(I)$, $\Lambda(B^{(k)})=\pi(J^{(k)})$ for
$k=1,2,\dots,m$, and $A=B^{(1)}+B^{(2)}+\cdots+B^{(m)}$ $($respectively,
$A\ge B^{(1)}+B^{(2)}+\cdots+B^{(m)}$$)$.
\end{enumerate}
\end{cor}
\begin{proof}
The trace identity \[
\sum_{\ell=1}^{s}(I'(\ell)-\ell)=\sum_{k=1}^{m}\sum_{\ell=1}^{s}(J^{\prime(k)}(\ell)-\ell)\]
shows that condition (\ref{equ-m-horn-ineq}) can be rewritten as\[
\sum_{\ell=1}^{s}(I\circ I'(\ell)-I'(\ell))\ge\sum_{k=1}^{m}\sum_{\ell=1}^{s}(J^{(k)}\circ J'^{(k)}(\ell)-J'^{(k)}(\ell)).\]
Therefore the $r$-tuples $-\pi(I),-\pi(J^{(1)}),\dots,-\pi(J^{(m)})$
satisfy condition (2) of Theorem \ref{Theo:Horn-conj-m}. Moreover,
(\ref{equation-horn-m-tuples}) is precisely the trace identity for
these tuples. The corollary follows easily from these observations.
\end{proof}
\begin{cor}
\label{cor:I-subset-of-first-n}Assume that $(I,J^{(1)},\dots,J^{(m)})\in\overline{T}_{r}^{N}(m+1)$
and $I\subset[n]$ for some $n<N$. Then we also have $J^{(k)}\subset[n]$
for $j=1,2,\dots,m$.
\end{cor}
It should be noted that the results in \cite{Ful-BAMS} are actually
formulated in symmetric form, i.e., the operator $A$ is replaced
by $B^{(0)}=-A$, so that the condition on these operators is $\sum_{k=0}^{m}B^{(k)}=0$
(or $\le0$). The passage from one formulation to the other is straightforward.
Denote indeed $\alpha=\Lambda(A),\beta^{(0)}=\Lambda(B^{(0)})$, and
observe that for every subset $I\subset\{1,2,\dots,N\}$ we have\[
\sum_{i\in I}\alpha_{i}=-\sum_{i\in I_{\text{sym}}}\beta_{i}^{(0)}.\]
 If $|I|=r$, we also have\[
\sum_{\ell=1}^{r}(I(\ell)-\ell)+\sum_{\ell=1}^{r}(I_{\text{sym}}(\ell)-\ell))=r(N-r).\]
This allows rewriting the definition of $T_{r}^{N}(m+1)$ in terms
of the sets $J^{(0)}=I_{\text{sym}},J^{(1)},\dots,J^{(m)}$. We prefer
however the less symmetric version of these results. One reason is
that we have $T_{r}^{N}(m+1)\subset T_{r}^{N+1}(m+1)$, and therefore
we obtain Horn inequalities which are valid for matrices of arbitrary
size. In fact, $T_{r}^{N}(m+1)$ consists precisely of those $(m+1)$-uples
in $T_{r}^{N+1}(m+1)$ which are contained in $[N]$. Another reason
is that we can assume that all the operators are positive.

The elements of $T_{r}^{N}(m+1)$ and $\overline{T}_{r}^{N}(m+1)$
also have a cohomological interpretation, related with the ring structure
of the cohomology of the Grassmannian $G(N,r)$ of $r$-dimensional
subspaces of $\mathbb{C}^{N}$. Since this connection was crucial
in the proof of the Horn conjecture, we describe it in more detail.
Fix subspaces $X_{1}\subset X_{2}\subset\cdots\subset X_{N}=\mathbb{C}^{N}$
with $\dim X_{j}=j$, and a subset $I=\{ i_{1}<i_{2}<\cdots<i_{r}\}$
of $[N]$. The associated Schubert cell\[
S=\{ M\in G(N,r):\dim(M\cap X_{i_{j}})\ge j\text{ for }j=1,2,\dots,r\}\]
determines a homology class $\eta_{I}\in H_{*}(G(N,r))$ which is
independent of the choice of $X_{j}$. Moreover, $H_{*}(G(N,r))$
is the free abelian group generated by the classes $\eta_{I}$ as
$I$ runs over all subsets $I\subset[N]$ with $|I|=r$. The particular
set $I=[r]$ corresponds with the class of one point. The cohomology
ring $H^{*}(G(N,r))$ has a dual basis $\omega_{I}$ indexed by $I\subset N$,
$|I|=r$, i.e., $\langle\eta_{I},\omega_{J}\rangle=\delta_{IJ}$.
The set $\overline{T}_{r}^{N}(m+1)$ consists of those $(m+1)$-tuples
$(I,J^{(1)},\dots,J^{(m)})$ such that the product\begin{equation}
\omega_{I}\omega_{J_{\text{sym}}^{(1)}}\cdots\omega_{J_{\text{sym}}^{(m)}}\label{eq:omega-product}\end{equation}
is not equal to zero, while $T_{r}^{N}(m+1)$ is characterized by
the fact that this product is a nonzero multiple of the class $\omega_{[r]}$
of one point. One can also identify a smaller class $\dot{T}_{r}^{N}(m+1)\subset T_{r}^{N}(m+1)$
corresponding with products (\ref{eq:omega-product}) which are exactly
equal to $\omega_{[r]}$. As seen in \cite{belk,Ful-BAMS}, the Horn
inequalities corresponding with $\dot{T}_{r}^{N}(m+1),r=1,2,\dots,N-1$,
are independent, and they imply (for decreasing sequences) all the
Horn inequalities corresponding with $T_{r}^{N}(m+1)$, or even $\overline{T}_{r}^{N}(m+1)$.
We will always formulate our results in terms of the sets $T$, but
generally one can use either $\overline{T}$ (if one wants to deduce
more inequalities) or $\dot{T}$ (if one wants minimal hypotheses).

We conclude this section with a few useful properties of the sets
$T$ and $\overline{T}$.

\begin{lem}
\label{lemma:inserting-gaps}Fix $(I,J^{(1)},J^{(2)},\dots,J^{(m)})\in T_{r}^{N}(m+1)$,
integers $0\le q,q_{1},\dots,q_{m}\le r$ such that $q=\sum_{k=1}^{m}q_{k}\le r$,
and an additional $M\in\mathbb{N}$. Denote by $I'$ the set obtained
by increasing the largest $q$ elements by $M$, i.e.,\[
I'(\ell)=\begin{cases}
I(\ell) & \text{if }\ell\le r-q,\\
I(\ell)+M & \text{if }\ell>r-q.\end{cases}\]
Analogously, define $J'^{(k)}$ by\[
J'^{(k)}(\ell)=\begin{cases}
J^{(k)}(\ell) & \text{if }\ell\le r-q_{k},\\
J^{(k)}(\ell)+M & \text{if }\ell>r-q_{k}.\end{cases}\]
 for $k=1,2,\dots,m$. Then $(I',J'^{(1)},J'^{(2)},\dots,J'^{(m)})\in T_{r}^{N+M}(m+1)$.
\end{lem}
\begin{proof}
By induction, it suffices to consider the case $q=1$. For simplicity,
assume that $q_{1}=1$ and $q_{k}=0$ for $k>1$. Let us set $\alpha=\pi(I),\alpha'=\pi(I')$,
$\beta_{k}=\pi(J^{(k)})$, and $\beta'_{k}=\pi(J'^{(k)})$ for $k=1,2,\dots r$.
By Corollary \ref{cor:Horn-indices-are-eigenvalues} (or by the definition
of $T_{r}^{N})$, the sequences $(\alpha,\beta^{(1)},\dots,\beta^{(k)})$
satisfy condition (1) of Theorem \ref{Theo:Horn-conj-m}, plus the
trace identity. Moreover, the differences $(\alpha'-\alpha,\beta'^{(1)}-\beta^{(1)},\dots,\beta'^{(k)}-\beta^{(k)})$
also satisfy condition (1) of Theorem \ref{Theo:Horn-conj-m} and
the trace identity. Indeed, $\alpha'-\alpha=\beta'^{(1)}-\beta^{(1)}=(M,0,\dots,0)$,
while $\beta'^{(k)}-\beta^{(k)}=0$ for $k>1$. We deduce that $(\alpha',\beta'^{(1)},\dots,\beta'^{(k)})$
satisfy condition (1) of Theorem \ref{Theo:Horn-conj-m} as well.
The conclusion follows from Corollary \ref{cor:Horn-indices-are-eigenvalues}. 
\end{proof}
\begin{prop}
\label{Propo:facts-about-Tnr}Fix integers $x<r<N$ and $y<N-r$.
\begin{enumerate}
\item If $(I,J^{(1)},\dots J^{(m)})\in\overline{T}_{r}^{N}(m+1)$ and $(I',J'^{(1)},\dots,J'^{(m)})\in\overline{T}_{x}^{r}(m+1)$
then $(I\circ I',J^{(1)}\circ J'^{(1)},\dots J^{(m)}\circ J'^{(m)})\in\overline{T}_{x}^{N}(m+1)$. 
\item If $(I,J^{(1)},\dots J^{(m)})\in\overline{T}_{r}^{N}(m+1)$ and $(I',J'^{(1)},\dots,J'^{(m)})\in\overline{T}_{y}^{N-r}(m+1)$
then $(I'',J''^{(1)},\dots,J''^{(m)})\in\overline{T}_{r+y}^{N}(m+1)$,
where $I''=I\cup(I^{c}\circ I')$ and $J''^{(k)}=J^{(k)}\cup(J^{(k)c}\circ J''^{(k)})$
for $k=1,2,\dots m$. 
\item If $(I,J^{(1)},\dots J^{(m)})\in\overline{T}_{r}^{(m+1)N}(m+1)$ then\[
|I\cap[N]|+\sum|J^{(k)}\setminus[mN]|\le r.\]

\end{enumerate}
\end{prop}
\begin{proof}
Properties (1) and (2) are known; see, for instance Buch \cite[Lemma 1]{buch}.
To prove (3), consider mutually orthogonal projections $P^{(1)},P^{(2)},\dots,P^{(m)}$
of size $(m+1)N\times(m+1)N$ and of rank $N$, and set $P=P^{(1)}+\cdots+P^{(m)}$.
Observe that \[
\alpha=\Lambda(-P)=(\underbrace{0,\dots,0}_{N},\underbrace{-1,\dots,-1}_{mN}),\quad\beta^{(k)}=\lambda(-P^{(k)})=(\underbrace{0,\dots,0}_{mN},\underbrace{-1,\dots,-1}_{N})\]
and therefore\[
\sum_{i\in I}\alpha_{i}=-|I\setminus[N]|,\quad\sum_{j\in J^{(k)}}\beta_{j}^{(k)}=-|J^{(k)}\setminus[mN]|.\]
Thus the Horn inequality applied to the eigenvalues of $-P,-P^{(1)},\dots,-P^{(m)}$
yields\[
-|I\setminus[N]|\le-\sum_{k=1}^{m}|J^{(k)}\setminus[mN]|.\]
The desired inequality is obtained since $|I\setminus[N]|=r-|I\cap[N]|$.
\end{proof}

\section{An Interpolation Result}

Given three vectors $\alpha,\alpha'\alpha''\in\mathbb{R}^{N}$, we
will say that $\alpha$ is \emph{between $\alpha'$} and $\alpha''$
if\[
\min\{\alpha_{i}',\alpha_{i}''\}\le\alpha_{i}\le\max\{\alpha_{i}',\alpha_{i}''\},\quad i=1,2,\dots,N.\]
We denote by $\alpha^{*}$ the decreasing rearrangement of a vector
$\alpha\in\mathbb{R}^{N}$.

\begin{lem}
Let $\alpha,\alpha',\alpha''\in\mathbb{R}^{N}$ be such that $\alpha'$
and $\alpha''$ are decreasing, and $\alpha$ is between $\alpha'$
and $\alpha''$. Then $\alpha^{*}$ is also between $\alpha'$ and
$\alpha''$.
\end{lem}
\begin{proof}
Replacing $\alpha'_{i}$ by $\min\{\alpha_{i}',\alpha_{i}''\}$ and
$\alpha_{i}''$ by $\max\{\alpha_{i}',\alpha_{i}''\}$, we may assume
that $\alpha'\le\alpha\le\alpha''$. We have then $\alpha'^{*}\le\alpha^{*}\le\alpha''^{*}$,
and clearly $\alpha'^{*}=\alpha',\alpha''^{*}=\alpha''$.
\end{proof}
We are now ready to interpolate between $N$-tuples which satisfy
the Horn inequalities and the reverse Horn inequalities. If $\alpha\in\mathbb{R}^{N}$,
and $I=\{ i_{1}<i_{2}<\cdots<i_{r}\}\subset[N]$, we will use the
notation $\alpha\circ I$ for the vector $(\alpha_{i_{1}},\alpha_{i_{2}},\dots,\alpha_{i_{r}})\in\mathbb{R}^{r}$.

\begin{prop}
\label{prop:interpolation-finiteN}Fix $m,N\in\mathbb{N}$, and consider
decreasing vectors $\alpha',\alpha'',\beta'^{(k)},\beta''^{(k)}\in\mathbb{R}^{N}$
for $k=1,2,\dots m$. Assume that for every $r\le N$ and every $(I,J^{(1)},\dots,J^{(m)})\in T_{r}^{N}(m+1)$
the inequalities\[
\sum_{i\in I}\alpha'_{i}\le\sum_{k=1}^{m}\sum_{j\in I^{(k)}}\beta{}_{j}^{\prime(k)}\]
and \[
\sum_{i\notin I}\alpha''_{i}\ge\sum_{k=1}^{m}\sum_{j\notin J^{(k)}}\beta_{j}^{\prime\prime(k)}\]
are satisfied. Then there exist Hermitian $N\times N$-matrices $A,B^{(1)},\dots,B^{(m)}$
such that $A=\sum_{k=1}^{m}B^{(k)}$, $\Lambda(A)$ is between $\alpha'$
and $\alpha''$, and $\Lambda(B^{(k)})$ is between $\beta'^{(k)}$
and $\beta''^{(k)}$ for $k=1,2,\dots,m$. If all the entries of $\alpha',\alpha'',\beta'^{(k)},\beta''^{(k)}$
are integers, then $A$ and $B^{(k)}$ can be chosen so that $\Lambda(A)$
and $\Lambda(B^{(k)})$ have integer entries as well. 
\end{prop}
\begin{proof}
We argue by induction on $N$, taking $N=0$ as our starting point
for which there is nothing to prove. Assume therefore that the proposition
has been proved for vectors in $\mathbb{R}^{M}$ with $M<N$. Let
us set $\alpha(t)=t\alpha'+(1-t)\alpha''$ and $\beta^{(k)}(t)=t\beta'^{(k)}+(1-t)\beta''^{(k)}$
for $t\in[0,1]$. Define $S\subset[0,1]$ to consist of those values
$t$ for which all the Horn inequalities\[
\sum_{i\in I}\alpha(t)_{i}\le\sum_{k=1}^{m}\sum_{j\in J^{(k)}}\beta^{(k)}(t)_{j}\]
 are satisfied. The set $S$ is clearly compact, $1\in S$, and therefore
it contains a smallest element $\tau$. We claim that there exist
$r\in[N]$ and $(I,J^{(1)},\dots,J^{(k)})\in T_{r}^{N}(m+1)$ such
that \begin{equation}
\sum_{i\in I}\alpha(\tau)=\sum_{k=1}^{m}\sum_{j\in J^{(k)}}\beta^{(k)}(\tau)_{j}.\label{eq-h-eq}\end{equation}
Indeed, if all of these inequalities were strict, $\tau$ would be
either zero or an interior point of $A$. The case $\tau=0$ is easily
dealt with because in this case $\alpha(\tau)=\alpha''$, and the
hypothesis implies\[
\sum_{i=1}^{N}\alpha(\tau)_{i}\ge\sum_{k=1}^{m}\sum_{j=1}^{N}\beta(\tau)_{j}^{(k)}.\]
Thus $\alpha(\tau),\beta^{(k)}(\tau)$ satisfy all the Horn inequalities,
as well as the trace identity, and the proposition follows from Theorem
\ref{Theo:Horn-conj-m} with $\alpha=\alpha(\tau)$ and $\beta^{(k)}=\beta^{(k)}(\tau)$.
Assume then that $\tau>0$. Proposition \ref{Propo:facts-about-Tnr}(1)
implies then the existence of $r\times r$-matrices $A_{0},B_{0}^{(1)},\dots,B_{0}^{(m)}$
such that $A_{0}=\sum_{k=1}^{m}B_{0}^{(k)}$, $\Lambda(A_{0})=\alpha(\tau)\circ I$,
and $\Lambda(B_{0}^{(k)})=\beta^{(k)}(\tau)\circ J_{0}^{(k)}$. We
claim that the vectors $\alpha'_{1}=\alpha(\tau)\circ I_{N-r}^{c},\beta{}_{1}^{\prime(k)}=\beta^{(k)}(\tau)\circ J_{N-r}^{(k)c}$,
$\alpha''_{1}=\alpha''\circ I_{N-r}^{c}$, $\beta_{1}^{\prime\prime(k)}=\beta''\circ J_{N-r}^{(k)c}$
satsify the hypothesis of the proposition, with $N-r$ in place of
$N$. Indeed, the required Horn inequalities for $\alpha'_{1}$ and
$\beta_{1}^{\prime(k)}$ follow from Proposition \ref{Propo:facts-about-Tnr}(2)
(the equality (\ref{eq-h-eq}) must be subtracted from the corresponding
Horn inequality for $\alpha(\tau)$ and $\beta^{(k)}(\tau)$). The
required Horn inequalities for $\alpha''_{1}$ and $\beta_{1}^{\prime\prime(k)}$
follow directly from Proposition \ref{Propo:facts-about-Tnr}(1).
The inductive hypothesis implies the existence of $(N-r)\times(N-r)$-matrices
$A_{1},B_{1}^{(k)}$ such that $A_{1}=\sum_{k=1}^{m}B_{1}^{(k)}$,
$\Lambda(A_{1})$ is between $\alpha'_{1}$ and $\alpha''_{1}$, and
$\Lambda(B_{1}^{(k)})$ is between $\beta_{1}^{\prime(k)}$ and $\beta_{1}^{\prime\prime(k)}$.
The conclusion of the proposition is satisfied by the matrices $A=A_{0}\oplus A_{1}$
and $B^{(k)}=B_{0}^{(k)}\oplus B_{1}^{(k)}$, $k=1,2,\dots,m$. The
verification of this fact is an easy application of the preceding
lemma because $\Lambda(A)=(\Lambda(A_{0}),\Lambda(A_{1}))^{*}$.

Assume now that the entries of $\alpha',\alpha'',\beta'^{(k)},\beta''^{(k)}$
are integers for $k=1,2,\dots,m$. Unfortunately, $\alpha(\tau)$
does not generally have integer entries. The argument will proceed
however in a similar manner. We construct for $n=0,1,\dots$ decreasing
integer vectors $\alpha(n),\beta^{(k)}(n)$ such that 
\begin{enumerate}
\item $\alpha(0)=\alpha',\beta^{(k)}(0)=\beta'^{(k)},$
\item $\alpha(n+1)$ is between $\alpha(n)$ and $\alpha''$, and $\beta^{(k)}(n+1)$
is between $\beta^{(k)}(n)$ and $\beta''^{(k)}$ for $k=1,2,\dots,m,$
\item $\sum_{i\in I}\alpha(n)_{i}\le\sum_{k=1}^{m}\sum_{j\in J^{(k)}}\beta^{(k)}(n)_{j}$
for every $(I,J^{(1)},\dots,J^{(m)})\in T_{r}^{N}(m+1)$, $r\le N$,
and
\item $\sum_{i=1}^{N}|\alpha(n+1)_{i}-\alpha(n)_{i}|+\sum_{k=1}^{m}\sum_{j=1}^{N}|\beta^{(k)}(n+1)_{j}-\beta^{(k)}(n)_{j}|=1.$
\end{enumerate}
In other words, only one entry of one of the vectors is modified in
passing from $n$ to $n+1$. The construction proceeds by induction
until either $\alpha(n)=\alpha''$ and $\beta^{(k)}(n)=\beta''$ for
all $n$, or one of the inequalities in (3) is an equality. The remainder
of the argument proceeds as before.

\end{proof}
In the following statement, the inequality $I'\le I$ is simply an
inequality between functions, i.e., $I'(\ell)\le I(\ell)$ for $\ell=1,2,\dots r$.

\begin{cor}
\label{corolar:from-Tbar-toT}Given $r<N$ and $(I,J^{(1)},\dots,J^{(m)})\in\overline{T}_{r}^{N}(m+1)$,
there exists $(I',J'^{(1)},\dots,J'^{(m)})\in T_{r}^{N}(m+1)$ such
that $I'\le I$ and $J'^{(k)}\ge J^{(k)}$ for $k=1,2,\dots,m$.
\end{cor}
\begin{proof}
Let us define partitions of length $r$ as follows: $\alpha''=\pi(I),$
$\beta''^{(k)}=\pi(J^{(k)}),$ $\alpha'=(0,0,\dots,0)$, and $\beta'^{(k)}=(N-r,N-r,\dots,N-r)$.
We claim that these vectors satisfy the hypotheses of Proposition
\ref{prop:interpolation-finiteN}. Indeed, the inequalities for $\alpha'',\beta''^{(k)}$
follow from the fact that $(I,J^{(1)},\dots,J^{(m)})\in\overline{T}_{r}^{N}(m+1)$,
and the Horn inequalities for $\alpha',\beta'^{(k)}$ are obvious.
Proposition \ref{prop:interpolation-finiteN} provides partitions
$\alpha=\Lambda(A),\beta^{(k)}=\Lambda(B^{(k)})$ which satisfy the
Horn inequalities and the trace identity, and such that \[
\alpha'\le\alpha\le\alpha''\quad\text{and}\quad\beta''^{(k)}\le\beta\le\beta'^{(k)}\]
for $k=1,2,\dots m$. To conclude the proof, we define $I',J'^{(k)}\subset[N]$
such that $\pi(I')=\alpha$ and $\pi(J'^{(k)})=\beta^{(k)}$ for $k=1,2,\dots,m$.
\end{proof}

\section{Eigenvalues of Compact Selfadjoint Operators}

Let $A$ be a compact selfadjoint operator on a complex Hilbert space
$\mathfrak{H}$. One can represent $A$ as\[
Ah=\sum_{k}\mu_{k}\langle h,e_{k}\rangle e_{k},\quad h\in\mathfrak{H},\]
where $\{ e_{k}\}$ is an orthonormal system, and $(\mu_{k})$ is
a real sequence with limit zero. It may no longer be possible to rearrange
the eigenvalues $\mu_{k}$ in decreasing order. Instead, we can define
numbers $\lambda_{\pm n}$ for $n\in\mathbb{N}$ such that $\lambda_{n}$
is th $n$th largest positive term in $(\mu_{k})$, while $\lambda_{-n}$
is the $n$th smallest negative term in $(\mu_{k})$. Note that\[
\lambda_{1}\ge\lambda_{2}\ge\cdots\ge0\text{ and }\lambda_{-1}\le\lambda_{-2}\le\cdots\le0.\]
 If there are only $p<\infty$ positive terms in $(\mu_{k})$, we
set $\lambda_{n}=0$ for $n>p$, with a similar convention for the
negative terms. We will denote by $\Lambda_{0}(A)$ the sequence\[
\lambda_{1}\ge\lambda_{2}\ge\cdots\ge\lambda_{-2}\ge\lambda_{-1}.\]
Observe that $A$ cannot quite be reconstructed, up to unitary equivalence,
from $\Lambda_{0}(A)$. It may happen that $0$ is not an eigenvalue
of $A$, but it figures infinitely many times in $\Lambda_{0}(A)$.
Conversely, $0$ may be an eigenvalue of $A$ but $\lambda_{\pm n}\ne0$
for all $n$. We will also use the notation $\Lambda_{+}(A)$ for
the sequence $\{\lambda_{n}\}_{n=1}^{\infty}$. Thus $\Lambda_{0}(A)$
can be identified with the pair $(\Lambda_{+}(A),-\Lambda_{+}(-A))$.

We will denote by $c_{\downarrow0\uparrow}$ the collection of all
real sequences $(\alpha_{\pm n})_{n=1}^{\infty}$ such that\[
\alpha_{1}\ge\alpha_{2}\ge\cdots\ge0\ge\cdots\ge\alpha_{-2}\ge\alpha_{-1}\]
and $\lim_{n\to\infty}\alpha_{\pm n}=0$. If $\alpha\in c_{\downarrow0\uparrow}$,
we will denote by $\overline{\alpha}$ the sequence $(\beta_{\pm n})_{n=1}^{\infty}$
defined by $\beta_{k}=-\alpha_{-k}$ for all $k=\pm n$. With this
notation, we have $\Lambda_{0}(-A)=\overline{\Lambda_{0}(A)}$ for
all compact selfadjoint operators $A$. 

We will prove analogues of the Horn inequalities by compressing operators
to finite-dimensional subspaces. The following observation shows why
this is possible.

\begin{lem}
Let $A$ be a compact selfadjoint operator on $\mathfrak{H}$, let $P$
be an orthogonal projection on $\mathfrak{H}$, and set $\alpha=\Lambda_{0}(A),\beta=\Lambda_{0}(PAP|P\mathfrak{H})$.
Then we have \[
\alpha_{n}\ge\beta_{n}\ge\beta_{-n}\ge\alpha_{-n},\quad n\in\mathbb{N}.\]

\end{lem}
\begin{proof}
This is an immediate consequence of the usual variational formulas\begin{equation}
\alpha_{n}=\inf_{\dim\mathfrak{M}<n}\sup\{\langle Ah,h\rangle:h\in\mathfrak{M}^{\perp},\| h\|=1\},\quad n\in\mathbb{N},\label{Rayleigh}\end{equation}
and their analogues for negative $n$.
\end{proof}
If $A$ is an $N\times N$ matrix, we can use both $\Lambda(A)$ and
$\Lambda_{0}(A)$ to describe the eigenvalues of $A$. Let's say that
$\Lambda(A)=\alpha\in\mathbb{R}^{N}$ and $\Lambda_{0}(A)=\beta\in c_{\downarrow0\uparrow}$.
We will have then $\alpha_{n}=\beta_{n}$ if $\alpha_{n}>0$, and
$\alpha_{n}=\beta_{n-N-1}$ if $\alpha_{n}<0$.

\begin{prop}
\label{Prop:infinite-horn}Fix compact selfadjoint operators $A,B^{(1)},B^{(2)},\dots,B^{(m)}$
on $\mathfrak{H}$ such that $A\le\sum_{k=1}^{m}B^{(k)}$, a tuple $(I,J^{(1)},J^{(2)},\dots,J^{(m)})\in T_{r}^{N}(m+1)$,
and nonnegative integers $q_{1},q_{2},\dots,q_{m}$ such that $q=\sum_{k=1}^{m}q_{k}\le r$.
The sequences $\alpha=\Lambda_{0}(A)$, and $\beta^{(k)}=\Lambda_{0}(B^{(k)})$
satisfy the inequality\begin{equation}
\sum_{\ell=1}^{r-q}\alpha_{I(\ell)}+\sum_{\ell=r-q+1}^{r}\alpha_{I(\ell)-N-1}\le\sum_{k=1}^{m}\left(\sum_{\ell=1}^{r-q_{k}}\beta_{J^{(k)}(\ell)}^{(k)}+\sum_{\ell=r-q_{k}+1}^{r}\beta_{J^{(k)}(\ell)-N-1}^{(k)}\right).\label{eq:gen-horn-ineq}\end{equation}

\end{prop}
\begin{proof}
If $P$ is an orthogonal projection of finite rank, we have $PAP\le\sum_{k=1}^{m}PB^{(k)}P$,
and therefore the vectors $\Lambda(PAP),\Lambda(PB^{(k)}P)$ must
satisfy the Horn inequalities. Let us choose $P$ so that its range
contains all the eigenvectors of $A$ and $B^{(k)}$ corresponding
with the eigenvalues $\alpha_{\pm n}$ and $\beta_{\pm n}^{(k)}$
for $n\le N$. If we denote $\alpha'=\Lambda_{0}(PAP|P\mathfrak{H})$
and $\beta'^{(k)}=\Lambda_{0}(PB^{(k)}P|P\mathfrak{H})$, the preceding
lemma shows that $\alpha'_{\pm n}=\alpha_{\pm n}$ and $\beta_{\pm n}^{\prime(k)}=\beta_{\pm n}$
for $n\le N$. Assume that the rank of $P$ is $N+M$, and let $(I',J'^{(1)},\dots,J'^{(m)})\in T_{r}^{N+M}(m+1)$
be constructed as in Lemma \ref{lemma:inserting-gaps}. This new $(m+1)$-tuple
provides a Horn inequality for $PAP|P\mathfrak{H}$ and $PB^{(k)}P|P\mathfrak{H}$,
and this inequality is precisely the conclusion of the proposition.
\end{proof}
The inequalities (\ref{eq:gen-horn-ineq}) will be referred to as
the \emph{extended Horn inequalities.}

The converse of the preceding result is also true. We will actually
prove a stronger result which also yields an analogue of Horn's conjecture.
In order to simplify the statement of the result, we will say that
an $(m+1)$-tuple $(\alpha,\beta^{(1)},\dots,\beta^{(k)})$ of sequences
in $c_{\downarrow0\uparrow}$ satisfies all the Horn inequalities
if they satisfy the conclusion of Proposition \ref{Prop:infinite-horn}
for every $r\le N$ and every $(I,J^{(1)},J^{(2)},\dots,J^{(m)})\in T_{r}^{N}(m+1)$
and every choice of $q_{1},q_{2},\dots,q_{m}$.

For the following preliminary result, it is useful to endow $c_{\downarrow0\uparrow}$
with the topology of uniform convergence. A particular case of this
result appears as \cite[Lemma 5.2]{BLS}. We include a proof for completeness.

\begin{lem}
\label{Lemma:closure-of-lambdas}For each $n\in\mathbb{N}$, let $A(n),B^{(1)}(n),\dots,B^{(m)}(n)$
be compact selfadjoint operators such that $A(n)=\sum_{k=1}^{m}B^{(k)}(n)$.
Assume that $\lim_{n\to\infty}\Lambda_{0}(A(n))=\alpha$ and $\lim_{n\to\infty}\Lambda_{0}(B^{(k)}(n))=\beta^{(k)}$
for $k=1,2,\dots,m$. Then there exist compact selfadjoint operators
$A,B^{(1)},\dots,B^{(m)}$ such that $A=\sum_{k=1}^{m}B^{(k)}$, $\Lambda_{0}(A)=\alpha$,
and $\Lambda_{0}(B^{(k)})=\beta^{(k)}$ for $k=1,2,\dots,m$.
\end{lem}
\begin{proof}
Set $\alpha(n)=\Lambda_{0}(A(n))$ and $\beta^{(k)}(n)=\Lambda_{0}(B^{(k)}(n))$
for $k=1,2,\dots,m$. Assume that al the given operators act on the
same Hilbert space $\mathfrak{H}$ with orthonormal basis $(e_{j})_{j=1}^{\infty}$.
Denote by $\mathfrak{H}_{i}$, $i\in\mathbb{N}$, the linear space generated
by $\{ e_{j}:1\le j\le3(m+1)i\}$. Replacing, if necessary, the operators
$A(n)$ and $B^{(k)}(n)$ by $U(n)A(n)U(n)^{*}$ and $U(n)B^{(k)}(n)U(n)^{*}$
for some unitary $U(n)$, we may assume that $\mathfrak{H}_{i}$ contains
the eigenvectors of $A(n)$ corresponding with eigenvalues $\alpha(n)_{\pm j}$,
$j=1,\dots,i$, and the eigenvectors of $B^{(k)}(n)$ corresponding
with eigenvalues $\beta^{(k)}(n)_{\pm j}$, $j=1,2,\dots,i$. (The
space $\mathfrak{H}_{i}$ has sufficiently large dimension to also accomodate
$i$ zero eigenvectors for each of the operators $A(n)$ and $B^{(k)}(n)$,
in case such eigenvectors do exist.) Note that $\| A(n)-P_{\mathfrak{H}_{i}}A(n)P_{\mathfrak{H}_{i}}\|\le\max\{|\alpha(n)_{i+1}|,|\alpha(n)_{-i-1}|\}\to0$
as $i\to\infty$, and the convergence assumption also implies that
the norms $\| A(n)\|$ are bounded. We deduce that the sequence $(A(n))_{n=1}^{\infty}$
has a subsequence which converges in norm. A similar argument applies
to $B^{(k)}(n)$. Thus, dropping to a subsequence, we may assume that
the limits $A=\lim_{n\to\infty}A(n)$ and $B^{(k)}=\lim_{n\to\infty}B^{(k)}(n)$
exist. Clearly then $A=\sum_{k=1}^{m}B^{(k)}$, and the variational
formulas (\ref{Rayleigh}) imply that $\alpha=\lim_{n\to\infty}\Lambda_{0}(A(n))=\Lambda_{0}(A)$
and similarly $\beta^{(k)}=\Lambda_{0}(B^{(k)})$.
\end{proof}
We come now to the central result of this paper.

\begin{thm}
\label{Main-result}Fix $m\in\mathbb{N}$, and consider sequences
$\alpha',\alpha'',\beta'^{(k)},\beta''^{(k)}\in c_{\downarrow0\uparrow}$
for $k=1,2,\dots m$. Assume that both $(\alpha',\beta^{'(1)},\dots,\beta'^{(k)})$
and $(\overline{\alpha''},\overline{\beta''^{(1)}},\dots,\overline{\beta''^{(m)}})$
satisfy all the extended Horn inequalities $(\ref{eq:gen-horn-ineq})$.
Then there exist compact selfadjoint operators $A,B^{(1)},\dots,B^{(m)}$
such that $A=\sum_{k=1}^{m}B^{(k)}$, $\Lambda_{0}(A)$ is between
$\alpha'$ and $\alpha''$, and $\Lambda_{0}(B^{(k)})$ is between
$\beta'^{(k)}$ and $\beta''^{(k)}$ for $k=1,2,\dots,m$.
\end{thm}
\begin{proof}
For each $n\in\mathbb{N}$, we will construct decreasing vectors $\alpha'(n),$
$\alpha''(n),$ $\beta'^{(k)}(n),$ and $\beta''^{(k)}(n)$ in $\mathbb{R}^{(m+1)n}$
satisfying the hypothesis of Proposition \ref{prop:interpolation-finiteN}
(with $N=(m+1)n)$, and such that \[
\alpha'(n)_{i}=\alpha'_{i},\alpha''(n)_{i}=\alpha''_{i},\beta'^{(k)}(n)_{i}=\beta_{i}^{\prime(k)},\beta''^{(k)}(n)_{i}=\beta_{i}^{\prime\prime(k)},\]
for $i\le n$, while\[
\alpha'(n)_{i}=\alpha'_{i-N-1},\alpha''(n)_{i}=\alpha''_{i-N-1},\beta'^{(k)}(n)_{i}=\beta_{i-N-1}^{\prime(k)},\beta''^{(k)}(n)_{i}=\beta_{i-N-1}^{\prime\prime(k)},\]
 for $i>N-n$; here $N=(m+1)n$. Proposition \ref{prop:interpolation-finiteN}
will then imply the existence of selfadjoint operators $A(n),B^{(k)}(n)$
of rank $\le(m+1)n$ such that $A(n)=\sum_{k=1}^{m}B^{(k)}(n)$, $\Lambda_{0}(A(n))_{\pm i}$
is between $\alpha'_{\pm i}$ and $\alpha''_{\pm i}$ for $i\le n$,
and $\Lambda_{0}(B^{(k)}(n))_{\pm i}$ is between $\beta_{\pm i}^{\prime(k)}$
and $\beta_{\pm i}^{\prime\prime(k)}$ for $i\le n$. There exists
a sequence $n_{1}<n_{1}<\cdots$ such that $(\Lambda_{0}(A(n_{i})))_{i=1}^{\infty}$
and $(\Lambda_{0}(B^{(k)}(n_{i})))_{i=1}^{\infty}$ have limits $\alpha$
and $\beta^{(k)}$ in $c_{\downarrow0\uparrow}$. The conclusion of
the theorem is then reached by an application of the preceding lemma.

The remaining task is the construction of the vectors $\alpha'(n),\alpha''(n),\beta'^{(k)}(n)$
and $\beta''^{(k)}(n)$. They are defined so that $\alpha'(n)$ and
$\beta''^{(k)}(n)$ are as small as possible, while $\alpha''(n)$
and $\beta'^{(k)}(n)$ are as large as possible. More precisely,\[
\alpha'(n)_{i}=\begin{cases}
\alpha'_{i} & \text{if }i\le n\\
\alpha'_{i-N-1} & \text{if }i>n\end{cases}\]
and\[
\alpha''(n)_{i}=\begin{cases}
\alpha''_{i} & \text{if }i\le N-n\\
\alpha''_{i-N-1} & \text{if }i>N-n\end{cases}\]
with similar definitions for $\beta'^{(k)}(n)$ and $\beta''^{(k)}(n)$.
We verify next that $\alpha'(n)$ and $\beta'^{(k)}(n)$ satisfy the
Horn inequalities. Fix $(I,J^{(1)},\dots,J^{(k)})\in T_{r}^{(m+1)n}(m+1)$,
and set\[
p=|I\cap[n]|,\quad q_{k}=|I\setminus[mn]|,\quad q=q_{1}+q_{2}+\cdots+q_{k}.\]
Proposition \ref{Propo:facts-about-Tnr}(3) implies that $p+q\le r$.
In particular, $\alpha'$ and $\beta'^{(k)}$ must satisfy the inequality\[
\sum_{\ell=1}^{r-q}\alpha'_{I(\ell)}+\sum_{\ell=r-q+1}^{r}\alpha'_{I(\ell)-N-1}\le\sum_{k=1}^{m}\left(\sum_{\ell=1}^{r-q_{k}}\beta_{J^{(k)}(\ell)}^{\prime(k)}+\sum_{\ell=r-q_{k}+1}^{r}\beta_{J^{(k)}(\ell)-N-1}^{\prime(k)}\right),\]
with $N=(m+1)n$. The definition of $q_{k}$ shows that \[
\sum_{\ell=1}^{r-q_{k}}\beta_{J^{(k)}(\ell)}^{\prime(k)}+\sum_{\ell=r-q_{k}+1}^{r}\beta_{J^{(k)}(\ell)-N-1}^{\prime(k)}=\sum_{\ell=1}^{r}\beta'^{(k)}(n)_{J_{k}(\ell)}.\]
On the other hand\[
\sum_{\ell=1}^{r}\alpha'(n)_{I(\ell)}=\sum_{\ell=1}^{p}\alpha'_{I(\ell)}+\sum_{\ell=p+1}^{r}\alpha'_{I(\ell)-N-1}\le\sum_{\ell=1}^{r-q}\alpha'_{I(\ell)}+\sum_{\ell=r-q+1}^{r}\alpha'_{I(\ell)-N-1}.\]
The inequality is obtained because $r-q\ge p$, and hence some negative
terms in the left hand side are replaced by positive ones. We deduce
that \[
\sum_{\ell=1}^{r}\alpha'(n)_{I(\ell)}\le\sum_{k=1}^{m}\sum_{\ell=1}^{r}\beta'^{(k)}(n)_{J_{k}(\ell)},\]
 so that $\alpha'(n)$ and $\beta^{(k)}(n)$ satisfy all the applicable
Horn inequalities. The same argument applied to $\overline{\alpha''}$
and $\overline{\beta'^{(k)}}$, along with the equivalence of (2)
and (3) in Theorem \ref{Theo:Horn-conj-m}, shows that all the hypotheses
of Proposition \ref{prop:interpolation-finiteN} are satisfied. The
proof of the theorem is complete.
\end{proof}
Our version of the Horn conjecture follows easily by taking $\alpha'=\alpha''$
and $\beta'^{(k)}=\beta''^{(k)}$.

\begin{thm}
\label{Cor:Horn-conj}Consider sequences $\alpha,\beta^{(1)},\dots,\beta^{(m)}\in c_{\downarrow0\uparrow}$.
The following conditions are equivalent:
\begin{enumerate}
\item There exist compact selfadjoint operators $A,B^{(1)},\dots,B^{(k)}$
such that $A=\sum_{k=1}^{m}B^{(k)}$, $\Lambda_{0}(A)=\alpha$, and
$\Lambda_{0}(B^{(k)})=\beta^{(k)}$ for $k=1,2,\dots,m$. 
\item The $(m+1)$-tuples $(\alpha,\beta^{(1)},\dots,\beta^{(k)})$ and
$(\overline{\alpha},\overline{\beta^{(1)}},\dots,\overline{\beta^{(k)}})$
satisfy all the extended Horn inequalities $(\ref{eq:gen-horn-ineq})$.
\end{enumerate}
\end{thm}
If only one half of the hypothesis of the theorem is satisfied, we
obtain an inequality in place of an equality.

\begin{cor}
\label{cor:Horn-conj-with-ineq}Fix $m\in\mathbb{N}$, and consider
sequences $\alpha,\beta^{(k)}\in c_{\downarrow0\uparrow}$ for $k=1,2,\dots m$.
Assume that $(\alpha',\beta^{'(1)},\dots,\beta'^{(k)})$ satisfy all
the extended Horn inequalities $(\ref{eq:gen-horn-ineq})$. Then there
exist compact selfadjoint operators $A,B^{(1)},\dots,B^{(m)}$ such
that $A\le\sum_{k=1}^{m}B^{(k)}$, $\Lambda_{0}(A)=\alpha$, and $\Lambda_{0}(B^{(k)})=\beta^{(k)}$
for $k=1,2,\dots,m$.
\end{cor}
\begin{proof}
Let us set $\beta'^{(k)}=\beta''^{(k)}=\beta^{(k)}$, $\alpha'=\alpha$,
and find a sequence $\alpha''\ge\alpha$ such that the hypotheses
of Theorem \ref{Main-result} are satisfied. We obtain operators $A',B^{(k)}$
such that $A'=\sum_{k=1}^{(m)}B^{(k)}$, $\Lambda_{0}(B^{(k)})=\beta^{(k)}$,
and $\Lambda_{0}(A')\ge\alpha$. To conclude the proof, choose an
operator $A\le A'$ with $\Lambda_{0}(A)=\alpha$.
\end{proof}
\begin{cor}
\label{cor--permutation}The sequences $\alpha,\beta^{(1)},\dots,\beta^{(m)}\in c_{\downarrow0\uparrow}$
satisfy all the extended Horn inequalities if and only if $\overline{\beta^{(1)}},\overline{\alpha},\beta^{(2)},\dots,\beta^{(k)}$
satisfy all the extended Horn inequalities.
\end{cor}
\begin{proof}
This is a simple consequence of the fact that $A\le\sum_{k=1}^{m}B^{(k)}$
if and only if $-B^{(k)}\le-A+\sum_{k=2}^{m}B^{(k)}$. Alternatively,
one can verify that the two requirements on the given sequences are
in fact identical.
\end{proof}

\section{Partially Specified Eigenvalues}

In this section we will investigate under what conditions we can find
operators $A=\sum_{k=1}^{m}B^{(k)}$ such that $\Lambda_{0}(A)$ and
$\Lambda_{0}(B^{(k)})$ are only partially specified. We start with
the matrix case, which is easier to handle. A decreasing function
$\alpha$ defined on a subset of $[N]$ will be called a partially
specified decreasing vector in $\mathbb{R}^{N}$. Fix such a partially
specified vector $\alpha=(\alpha_{i_{1}}\ge\alpha_{i_{2}}\ge\cdots\ge\alpha_{i_{p}})$,
where $i_{1}<i_{2}<\cdots<i_{p}$. We define vectors $\alpha^{\min}$
and $\alpha^{\max}$ as follows:\[
\alpha_{i}^{\min}=\begin{cases}
\alpha_{i_{1}} & \text{if }i\le i_{1}\\
-\infty & \text{if }i_{p}<i\le N\\
\alpha_{i_{j+1}} & \text{if }i_{j}<i\le i_{j+1}\end{cases},\quad\alpha_{i}^{\max}=\begin{cases}
+\infty & \text{if }i<i_{1}\\
\alpha_{i_{p}} & \text{if }i_{p}\le i\le N\\
\alpha_{i_{j}} & \text{if }i_{j}\le i<i_{j+1}\end{cases}.\]
A decreasing vector $\beta\in\mathbb{R}^{N}$ agrees with $\alpha$
on the specified indices if and only if $\alpha^{\min}\le\beta\le\alpha^{\max}$;
we will write $\beta\supset\alpha$ when this happens.

In the following statement, the various vectors need not be specified
on the same collection of indices.

\begin{prop}
\label{pro:partial-matrixCase}Fix $m,N\in\mathbb{N}$, and consider
partially specified decreasing vectors $\alpha',\alpha'',\beta'^{(k)},\beta''^{(k)}$
in $\mathbb{R}^{N}$ for $k=1,2,\dots m$. The following conditions
are equivalent:
\begin{enumerate}
\item There exist Hermitian $N\times N$ matrices $A,B^{(1)},\dots,B^{(m)}$such
that $A=\sum_{k=1}^{m}B^{(k)}$, $\Lambda(A)\supset\alpha$, and $\Lambda(B^{(k)})\supset\beta^{(k)}$
for $k=1,2,\dots,m$.
\item For every $r\le N$ and every $(I,J^{(1)},\dots,J^{(m)})\in T_{r}^{N}(m+1)$
the inequalities\[
\sum_{i\in I}\alpha_{i}^{\min}\le\sum_{k=1}^{m}\sum_{j\in I^{(k)}}\beta_{j}^{(k)\max}\]
and \[
\sum_{i\notin I}\alpha_{i}^{\max}\ge\sum_{k=1}^{m}\sum_{j\notin J^{(k)}}\beta_{j}^{(k)\min}\]
are satisfied.
\end{enumerate}
\end{prop}
\begin{proof}
The implication $(1)\Rightarrow(2)$ follows immediately from the
fact that the vectors $\Lambda(A)$ and $\Lambda(B^{(k)})$ satisfy
all the Horn inequalities and the trace identity. Conversely, assume
that $(2)$ is satisfied. Given a large positive constant $C$, replace
all the $-\infty$ entries of $\alpha^{\min}$ by $-C$ to obtain
a vector $\alpha'$. Similarly, define $\beta'^{(k)}$ by replacing
the $\infty$ entries of $\beta^{(k)\max}$ by $C$. If $C$ is sufficiently
large, then the inequalities

\[
\sum_{i\in I}\alpha'_{i}\le\sum_{k=1}^{m}\sum_{j\in I^{(k)}}\beta{}_{j}^{\prime(k)}\]
will be satisfied for every $r\le N$ and every $(I,J^{(1)},\dots,J^{(m)})\in T_{r}^{N}(m+1)$.
Indeed, those inequalities which do not involve $C$ are satisfied
by hypothesis, and the remaining ones amount to a finite number of
restrictions which are satisfied for sufficiently large $C$. On the
other hand, replacing the $\infty$ entries of $\alpha^{\max}$ by
$C$ and the $-\infty$ entries of $\beta^{(k)\min}$ by $-C$, we
obtain decreasing sequences $\alpha''$and $\beta^{\prime\prime(k)}$
such that \[
\sum_{i\notin I}\alpha''_{i}\ge\sum_{k=1}^{m}\sum_{j\notin J^{(k)}}\beta_{j}^{\prime\prime(k)}\]
for every $r\le N$ and every $(I,J^{(1)},\dots,J^{(m)})\in T_{r}^{N}(m+1)$.
Proposition \ref{prop:interpolation-finiteN} yields then the desired
matrices.
\end{proof}
Observe that condition (2) above lists some inequalities with infinite
terms. More precisely, there may be $-\infty$ terms in the left-hand
side and/or $+\infty$ terms in the right-hand side. These inequalities
are automatically satisfied, so one only needs to consider those inequalities
where all the terms are finite, and therefore are equal to some of
the specified entries of the given vectors. To illustrate this result,
consider the case in which $\beta^{(1)},\dots,\beta^{(m)}$ are fully
specified decreasing vectors, but only $\alpha_{p}$ is specified.
In this case, \[
\alpha^{\min}=(\underbrace{\alpha_{p},\dots,\alpha_{p}}_{p},\underbrace{-\infty,\dots,-\infty}_{N-p})\]
and \[
\alpha^{\max}=(\underbrace{\infty,\dots,\infty}_{p-1},\underbrace{\alpha_{p},\dots,\alpha_{p}}_{N-p+1}).\]
The relevant upper estimates are then\[
r\alpha_{p}=\sum_{i\in I}\alpha_{i}^{\min}\le\sum_{k=1}^{m}\sum_{j\in J^{(k)}}\beta_{j}^{(k)},\]
where $(I,J^{(1)},\dots,J^{(m)})\in T_{r}^{N}(m+1)$ and $I\subset[p]$.
Many of these inequalities are in fact redundant. Thus, we must have
\[
\sum_{k=1}^{m}(J^{(k)}(r)-1)\le I(r)-1\le p-1,\]
and therefore\[
\sum_{k=1}^{m}(J^{(k)}(\ell)-1)\le p-1,\quad\ell=1,2,\dots,r.\]
It follows that all of these inequalities will be satisfied if\[
\alpha_{p}\le\min\left\{ \sum_{k=1}^{m}\beta_{j_{k}}^{(k)}:\sum_{k=1}^{m}(j_{k}-1)=p-1\right\} .\]
Similarly, the lower estimates amount to\[
\alpha_{p}\ge\max\left\{ \sum_{k=1}^{m}\beta_{j_{k}}^{(k)}:\sum_{k=1}^{m}(N-j_{k})=N-p\right\} .\]
This result was proved by Johnson \cite{Johnson}. for $m=2$.

These considerations can be extended to any situation where the vectors
$\beta^{(k)}$ are fully specified, thus answering a question raised
in \cite{Ful-BAMS}. Finding a minimal set of inequalities may require
some work. We illustrate this in case $\alpha_{1}$ and $\alpha_{3}$
are specified. For simplicity, assume that $m=2$, and set $\beta=\beta^{(1)}$,
$\gamma=\beta^{(2)}$. In this case\[
\alpha^{\min}=(\alpha_{1},\alpha_{3},\alpha_{3},\underbrace{-\infty,\dots,-\infty}_{N-3}),\quad\alpha^{\max}=(\alpha_{1},\alpha_{1},\underbrace{\alpha_{3},\dots,\alpha_{3}}_{N-2}).\]
 For the upper estimates, we need to consider Horn triples $(I,J,K)$
with $I\subset[3]$. When $|I|=1$, we obtain\[
\alpha_{1}\le\beta_{1}+\gamma_{1},\quad\alpha_{3}\le\min\{\beta_{1}+\gamma_{3},\beta_{2}+\gamma_{2},\beta_{3}+\gamma_{1}),\]
when $|I|=3$ we obtain the inequality \[
\alpha_{1}+2\alpha_{3}\le\sum_{j=1}^{3}(\beta_{j}+\gamma_{j}),\]
and when $|I|=2$ we have $I=\{1,2\}$ which yields\begin{equation}
\alpha_{1}+\alpha_{3}\le\beta_{1}+\beta_{2}+\gamma_{1}+\gamma_{2},\label{eq(53)}\end{equation}
$I=\{1,3\}$ which yields\[
\alpha_{1}+\alpha_{3}\le\min\{\beta_{1}+\beta_{3}+\gamma_{1}+\gamma_{2},\beta_{1}+\beta_{2}+\gamma_{1}+\gamma_{3}\},\]
and finally $I=\{2,3\}$ which yields\[
2\alpha_{3}\le\begin{cases}
\beta_{2}+\beta_{3}+\gamma_{1}+\gamma_{2}\\
\beta_{1}+\beta_{2}+\gamma_{2}+\gamma_{3}\\
\beta_{1}+\beta_{3}+\gamma_{1}+\gamma_{3}\end{cases}.\]
As before, these last three estimates and (\ref{eq(53)}) can be discarded
in order to obtain a minimal set of upper estimates. A minimal set
of lower estimates is harder to obtain. The inequalities for which
$[N]\setminus I\subset[2]$ or $[N]\setminus I\subset[2]^{c}$ involve
only one of the specified entries, and will be satisfied provided
that\[
\alpha_{1}\ge\max\{\beta_{j}+\gamma_{k}:j+k=N+1\},\quad\alpha_{3}\ge\max\{\beta_{j}+\gamma_{k}=N+3\}.\]
In addition, there will be lower bounds for $\alpha_{1}+(r-1)\alpha_{3}$
and $2\alpha_{1}+(r-2)\alpha_{3}$ when $3\le r\le N$. These are
provided by the reverse Horn inequalities with $[N]\setminus I=[r+1]\setminus\{2\}$
and $[N]\setminus I=[r]$, respectively. The fact that one does not
need to consider more general sets $I$ is deduced easily from Corollary
\ref{corolar:from-Tbar-toT}. Indeed, assume that we write a reverse
Horn inequality such that $\sum_{i\notin I}\alpha_{i}^{\max}=\alpha_{1}+(r-1)\alpha_{3}$.
Replacing $I$ by $I'$ such that $[N]\setminus I'=[r+1]\setminus\{2\}$
we have $\sum_{i\notin I}\alpha_{i}^{\max}=\sum_{i\notin I'}\alpha_{i}^{\max}$,
and the $(m+1)$-tuple $(I',J^{(1)},\dots,J^{(m)})$ belongs to $\overline{T}_{N-r}^{N}$.
Corollary \ref{corolar:from-Tbar-toT} yields then $(I'',J'^{(1)},\dots,J'^{(m)})$
in $T_{N-r}^{N}$ such that $I''\ge I'$ and $J'^{(k)}\ge J^{(k)}$.
The reverse Horn inequality for this new tuple will be stronger than
the one given by the original $(I,J^{(1)},\dots,J^{(m)})$. One reasons
in a similar fashion when $\sum_{i\notin I}\alpha_{i}^{\max}=2\alpha_{1}+(r-2)\alpha_{3}$.
Some of the inequalities with $[N]\setminus I=[r+1]\setminus\{2\}$
might not belong to $\dot{T}_{N-r}^{N}$ and can therefore be discarded.

The results of \cite{buch} can also be deduced from Proposition \ref{pro:partial-matrixCase}.
Namely, one considers the vectors $\beta^{(1)},\dots,\beta^{(m)}$
to be fully specified, while the only specified entries of $\alpha$
are $\alpha_{\rho+1}=\alpha_{\rho+2}=\cdots=\alpha_{N}=0$. This amounts
to looking at Hermitian matrices $B^{(1)},\dots,B^{(m)}$ with specified
eigenvalues, whose sum is positive and has rank at most $\rho$. The
conditions found in \cite{buch} form a minimal system of inequalities
equivalent to the ones provided by Proposition \ref{pro:partial-matrixCase}.
The upper estimates which must be kept are those corresponding to
$(I,J^{(1)},\dots,J^{(m)})\in\dot{T}_{r}^{N}$ such that $I=[N]\setminus[N-r]$,
as can be seen by an argument using Corollary \ref{corolar:from-Tbar-toT}.
The lower estimates correspond with $I\supset[\rho]$, and one only
needs to consider those tuples in $\dot{T}$.

We consider now compact selfadjoint operators with partially specified
eigenvalues. Let $S$ be a set of nonzero integers. We assume, for
simplicity, that $S\cap\mathbb{N}$ and $S\cap(-\mathbb{N})$ are
infinite. A function $n\mapsto\alpha_{n}\in\mathbb{R}$ defined on
$S$ will be called a partially specified element of $c_{\downarrow0\uparrow}$
if there exists $\beta\in c_{\downarrow0\uparrow}$ such that $\beta_{n}=\alpha_{n}$
for all $n\in S$. Let $\alpha$ be such a sequence, and assume $S\cap\mathbb{N}=\{ n_{1}<n_{2}<\cdots\}$,
$S\cap(-\mathbb{N})=\{ m_{1}>m_{2}>\dots\}$. We introduce two sequences
$(\alpha_{\pm n}^{\max})_{n=1}^{\infty}$ and $(\alpha_{\pm n}^{\min})_{n=1}^{\infty}$
as follows:\[
\alpha_{n}^{\max}=\begin{cases}
\infty & \text{if }0<n<n_{1}\\
\alpha_{n_{j}} & \text{if }n_{j}\le n<n_{j+1}\\
\alpha_{m_{1}} & \text{if }m_{1}\le n<0\\
\alpha_{m_{j+1}} & \text{if }m_{j+1}\le n<m_{j}\end{cases}\quad\alpha_{n}^{\min}=\begin{cases}
\alpha_{n_{1}} & \text{if }0<n\le n_{1}\\
\alpha_{n_{j+1}} & \text{if }n_{j}<n\le n_{j+1}\\
-\infty & \text{if }m_{1}<n<0\\
\alpha_{m_{j}} & \text{if }m_{j+1}<n\le m_{j}\end{cases}.\]
The elements $\beta\in c_{\downarrow0\uparrow}$ which extend $\alpha$
are precisely those satisfying the inequalities $\alpha^{\min}\le\beta\le\alpha^{\max}$.
We will use the notation $\beta\supset\alpha$ to indicate that the
sequence $\beta$ extends the partially defined $\alpha$. 

In the following result, as in its finte-dimensional counterpart,
an extended Horn inequality in which an infinite term appears is to
be automatically satisfied. In order to see what the difficulty is
in extending the proof of Proposition \ref{pro:partial-matrixCase},
observe that there are infinitely many extended Horn inequalities.
Thus, when replacing an infinite entry by a finite constant, we impose
infinitely many conditions on that constant. We need to show that
these constraints can be met simultaneously.

\begin{prop}
\label{prop:Buch-extended} Let $\alpha,\beta^{(1)},\beta^{(2)},\dots,\beta^{(m)}$
be partially specified elements of $c_{\downarrow0\uparrow}$. The
following conditions are equivalent:
\begin{enumerate}
\item There exist compact selfadjoint operators $A,B^{(1)},\dots,B^{(k)}$
such that $A=\sum_{k=1}^{m}B^{(k)}$, $\Lambda_{0}(A)\supset\alpha$,
and $\Lambda_{0}(B^{(k)})\supset\beta^{(k)}$ for $k=1,2,\dots,m$. 
\item Both $(\alpha^{\min},\beta^{(1)\max},\dots,\beta^{(m)\max})$ and
$(\overline{\alpha^{\max}},\overline{\beta^{(1)\min}},\dots,\overline{\beta^{(m)\min}})$
satisfy all the Horn inequalities.
\end{enumerate}
\end{prop}
\begin{proof}
The implication $(1)\Rightarrow(2)$ follows immediately from the
corresponding implication in Theorem \ref{Main-result}. Theorem \ref{Main-result}
also shows that, to prove the opposite implication, it will suffice
to find sequences $\alpha',\alpha'',\beta'^{(k)},\beta''^{(k)}\in c_{\downarrow0\uparrow}$
such that the $(m+1)$-tuples $(\alpha^{\prime},\beta'^{(1)},\dots,\beta'^{(k)})$
and $(\overline{\alpha^{\prime\prime}},\overline{\beta''^{(1)}},\dots,\overline{\beta''^{(k)}})$
satisfy all the Horn inequalities, $\alpha'\ge\alpha^{\min}$, $\alpha''\le\alpha^{\max}$,
$\beta'^{(k)}\le\beta^{(k)\max}$, and $\beta''^{(k)}\ge\beta^{(k)\min}$.
By symmetry, it will suffice to construct the sequences $\alpha'$
and $\beta^{\prime(k)}$. We start by constructing $\alpha'\ge\alpha^{\min}$
such that $(\alpha',\beta^{(1)\max},\dots,\beta^{(m)\max})$ satisfy
all the extended Horn inequalities. The vectors $\alpha^{\min}$ and
$\beta^{(k)\max}$ have finitely many infinite entries. Assume for
definiteness that $\alpha_{i}^{\min}=-\infty$ for $i=-1,-2,\dots,-u$,
$\beta_{i}^{\max}=+\infty$ for $i=1,2,\dots,v_{k}$, and all the
other entries of these sequences are finite. We will construct inductively
sequences $\alpha^{(0)}=\alpha,\alpha^{(1)},\alpha^{(2)},\dots,\alpha^{(v)}=\alpha'$
such that each $\alpha^{(i+1)}$ is obtained from $\alpha^{(i)}$
by replacing one $-\infty$ entry by a large negative number, and
$(\alpha^{(i)},\beta^{(1)},\dots,\beta^{(m)})$ satisfy all the extended
Horn inequalities. It will suffice to show that it is possible to
construct $\alpha^{(1)}$. Denote by $\alpha^{(1)}$ the sequence
defined by\[
\alpha_{i}^{(1)}=\begin{cases}
\alpha_{i}^{\min} & \text{if }i\ne-u\\
\tau & \text{if }i=-u\end{cases}\]
 for $\tau\le\alpha_{-u-1}^{\min}$. We verify that $(\alpha^{(1)},\beta^{(1)\max},\dots,\beta^{(m)\max})$
satisfies all the Horn inequalities if $\tau$ is sufficiently small.
The relevant inequalities which are not already covered by the hypothesis
of the propositions are\[
\sum_{\ell=1}^{r-q}\alpha_{I(\ell)}^{(1)}+\sum_{\ell=r-q+1}^{r}\alpha(\tau)_{I(\ell)-N-1}\le\sum_{k=1}^{m}\left(\sum_{\ell=1}^{r-q_{k}}\beta_{J^{(k)}(\ell)}^{(k)\max}+\sum_{\ell=r-q_{k}+1}^{r}\beta_{J^{(k)}(\ell)-N-1}^{(k)\max}\right),\]
where $(I,J^{(1)},\dots,J^{(k)})\in T_{r}^{N}(m+1),$ $q=q_{1}+\cdots+q_{m}<r$,
$I(r)-N-1=-u$, and $J^{(k)}(1)>v_{k}$ for $k=1,2,\dots,m$. The
$(m+1)$-tuple $(J,J,\dots,J)$, with $J=\{1,2,\dots,r-1\}$, belongs
to $T_{r-1}^{r}(m+1)$; this is an easy consequence of Corollary \ref{cor:Horn-indices-are-eigenvalues}
since $\pi(J)=0$. Therefore $(I\circ J,J^{(1)}\circ J,\dots,J^{(m)}\circ J)\in\overline{T}_{r-1}^{N}(m+1)$
by Proposition \ref{Propo:facts-about-Tnr}(1). Corollary \ref{corolar:from-Tbar-toT}
provides $(I',J'^{(1)},\dots,J'^{(k)})\in T_{r-1}^{N}(m+1)$ such
that $I'\le I\circ J$ and $J'^{(k)}\ge J^{(k)}\circ J$ for $k=1,2,\dots,m$.
The hypothesis implies then the inequality\[
\sum_{\ell=1}^{r-1-q'}\alpha_{I'(\ell)}^{\min}+\sum_{\ell=r-q'}^{r-1}\alpha_{I'(\ell)-N-1}^{\min}\le\sum_{k=1}^{m}\left(\sum_{\ell=1}^{r-1-q'_{k}}\beta_{J'^{(k)}(\ell)}^{(k)\max}+\sum_{\ell=r-q'_{k}}^{r-1}\beta_{J'^{(k)}(\ell)-N-1}^{(k)\max}\right),\]
and the relations between $I'$ and $I\circ J$, along with the fact
that $\alpha$ is decreasing, imply\[
\sum_{\ell=1}^{r-1-q'}\alpha_{I(\ell)}^{\min}+\sum_{\ell=r-q'}^{r-1}\alpha_{I(\ell)-N-1}^{\min}\le\sum_{k=1}^{m}\left(\sum_{\ell=1}^{r-1-q'_{k}}\beta_{J^{(k)}(\ell)}^{(k)\max}+\sum_{\ell=r-q'_{k}}^{r-1}\beta_{J^{(k)}(\ell)-N-1}^{(k)\max}\right).\]
With the choice $q'_{k_{0}}=q_{k_{0}}-1$ and $q'_{k}=q_{k}$ for
$k\ne k_{0}$, we obtain\[
\sum_{\ell=1}^{r-q}\alpha_{I(\ell)}^{\min}+\sum_{\ell=r-q+1}^{r-1}\alpha_{I(\ell)-N-1}^{\min}\le\sum_{k=1}^{m}\left(\sum_{\ell=1}^{r-q{}_{k}}\beta_{J^{(k)}(\ell)}^{(k)\max}+\sum_{\ell=r-q{}_{k}+1}^{r-1}\beta_{J^{(k)}(\ell)-N-1}^{(k)\max}\right)\]
after also replacing some negative terms by positive ones in the right-hand
side of the inequality. The desired inequality for $\alpha^{(1)}$
will then follow provided that\[
\tau\le\sum_{k=1}^{m}\beta_{J^{(k)}(r)-N-1}^{(k)\max},\]
and for this it will suffice that\[
\tau\le\sum_{k=1}^{m}\beta_{-1}^{(k)\max}.\]

The construction of $\beta^{\prime(k)}$ is then done by induction
on $k$ so that \[
(\alpha',\beta'^{(1)},\dots,\beta'^{(k)},\beta^{(k+1)\max},\dots,\beta^{(m)\max})\]
 satisfies all the extended Horn inequalities. For instance, we know
from Corollary \ref{cor--permutation} that $(\overline{\beta^{(1)\max}},\overline{\alpha'},\beta^{(2)\max},\dots,\beta^{(m)\max})$
satisfies all the extended Horn inequalities, and the preceding construction
(with $\overline{\beta^{(1)\max}}$ in place of $\alpha^{\min}$)
yields the desired sequence $\beta^{\prime(1)}$. This concludes the
construction of the sequences $\beta^{\prime(k)}$ and the proof of
the proposition.
\end{proof}

\section{Cases of Equality, Positive Operators}

In finite dimensions it is known \cite{horn} that, when one of the
Horn inequalities is an equality, there is a subspace which reduces
all the matrices involved. This result extends easily to compact selfadjoint
operators.

\begin{prop}
Let $A,B^{(1)},\dots,B^{(k)}$ be compact selfadjoint operators such
that $A=\sum_{k=1}^{m}B^{(k)}$, and set $\alpha=\Lambda_{0}(A)$
and $\beta^{(k)}=\Lambda_{0}(B^{(k)})$. Assume that the $(m+1)$-tuple
$(I,J^{(1)},\dots,J^{(m)})\in T_{r}^{N}(m+1)$, $q_{1},\dots,q_{m}$
are such that $q=\sum_{k=1}^{m}q_{k}\le r$, and \[
\sum_{\ell=1}^{r-q}\alpha_{I(\ell)}+\sum_{\ell=r-q+1}^{r}\alpha_{I(\ell)-N-1}=\sum_{k=1}^{m}\left(\sum_{\ell=1}^{r-q_{k}}\beta_{J^{(k)}(\ell)}^{(k)}+\sum_{\ell=r-q_{k}+1}^{r}\beta_{J^{(k)}(\ell)-N-1}^{(k)}\right).\]
Then there is a space $\mathfrak{M}$ of dimension $N$ which reduces
the operators $A$ and $B^{(k)}$, \[
\Lambda(A|\mathfrak{M})=(\alpha_{I(1)},\dots,\alpha_{I(r-q)},\alpha_{I(r-q+1)-N-1},\dots,\alpha_{I(r)-N-1}),\]
and similarly\[
\Lambda(B^{(k)}|\mathfrak{M})=(\beta_{J^{(k)}(1)}^{(k)},\dots,\beta_{J^{(k)}(r-q_{k})}^{(k)},\beta_{J^{(k)}(r-q_{k}+1)-N-1}^{(k)},\dots,\beta_{J^{(k)}(r)-N-1}^{(k)})\]
for $k=1,2,\dots,r$.
\end{prop}
\begin{proof}
Assume that the operators act on the separable space $\mathfrak{H}$.
Construct spaces $\mathfrak{H}_{n}$ such that $\mathfrak{\mathfrak{H}}_{n}\subset\mathfrak{H}_{n+1}$,
$\bigcup\mathfrak{H}_{n}$ is dense in $\mathfrak{H}$, and $\mathfrak{H}_{n}$
contains all eigenvectors for $A$ and $B^{(k)}$ corresponding with
the eigenvalues $\alpha_{I(\ell)},\alpha_{I(\ell)-N-1},\beta_{J^{(k)}(\ell)}^{(k)},\beta_{J^{(k)}(\ell)-N-1}^{(k)}$
for $\ell=1,2,\dots,r$, provided that $n\ge2r(m+1)$. The finite-dimensional
result implies the existence of spaces $\mathfrak{M}_{n}$, reducing for
$P_{\mathfrak{H}_{n}}A|\mathfrak{H}_{n}$ and $P_{\mathfrak{H}_{n}}B^{(k)}|\mathfrak{H}_{n}$,
such that $\Lambda(P_{\mathfrak{H}_{n}}A|\mathfrak{M}_{n})$ and $\Lambda(P_{\mathfrak{H}_{n}}B^{(k)}|\mathfrak{M}_{n})$
have the desired eigenvalues for $n$ large. If all of these eigenvalues
are different from zero, it will follow that all the spaces $\mathfrak{M}_{n}$
are contained in a fixed finite-dimensional space, and then we can
choose $\mathfrak{M}$ to be any limit point of this sequence. The case
in which some of the eigenvalues involved are zero requires minor
modifications which we leave to the interested reader.
\end{proof}
Corollary \ref{Cor:Horn-conj} specializes as follows in the case
of positive operators.

\begin{thm}
\label{thm:positive-horn}Let $\alpha,\beta^{(1)},\dots,\beta^{(m)}$
be decreasing sequences with limit zero. the following conditions
are equivalent:
\begin{enumerate}
\item There exist positive compact operators $A,B^{(1)},\dots,B^{(k)}$
such that $A=\sum_{k=1}^{m}B^{(k)}$, $\Lambda_{+}(A)=\alpha$, and
$\Lambda_{+}(B^{(k)})=\beta^{(k)}$ for $k=1,2,\dots,m$. 
\item For every $r<N$, every $q_{1},\dots,q_{m}\ge0$ such that $q=q_{1}+\cdots+q_{m}\le r$,
and every $(I,J^{(1)},\dots,J^{(m)})\in T_{r}^{N}(m+1)$, we have
the Horn inequality\[
\sum_{i\in I}\alpha_{i}\le\sum_{k=1}^{m}\sum_{j\in J^{(k)}}\beta_{j}^{(k)}\]
and the extended reverse Horn inequality\[
\sum_{i\in I_{q}^{c}}\alpha_{i}\ge\sum_{k=1}^{m}\sum_{j\in J_{q_{k}}^{(k)c}}\beta_{j}^{(k)}.\]

\end{enumerate}
\end{thm}
The lack of balance in the reverse Horn inequalities comes from the
fact that all negative terms in $\Lambda_{0}(A)$ and $\Lambda_{0}(B^{(k)})$
are equal to zero. The proof of the proposition is straighforward
if we take into account the equivalence between conditions (2) and
(3) in Theorem \ref{Theo:Horn-conj-m}. Note that the extended Horn
inequalities need not be required in full generality -- the ordinary
Horn inequalities suffice in the case of positive operators. Indeed,
applying the Horn inequalities provided by Lemma \ref{lemma:inserting-gaps}
one obtains the extended Horn inequalities as $M\to\infty$. This
is not the case for the reverse inequalities. While this observation
does reduce the collection of inequalities to be verified, the smaller
collection is still redundant. As noted in \cite{key-230,Ful-LAA},
one can delete any finite collection of Horn inequalities from condition
(2) without altering the conclusion of the theorem.

We will now illustrate some of our results with a simple example where
$m=2$. Consider the sequences $\beta=\gamma'=(1/2n)_{n=1}^{\infty}$,
and positive compact operators $B_{0}=C_{0}$ such that $\Lambda_{+}(B_{0})=\beta.$
Note that $\alpha=\Lambda_{+}(B_{0}+C_{0})=(1/n)_{n=1}^{\infty}$.
On the other hand, let $C_{1}$ be a positive compact operator such
that $\Lambda_{+}(C_{1})=\gamma''=(1/(2n-1))_{n=1}^{\infty}$. Then\[
\alpha=\Lambda_{+}(B_{0}\oplus C_{1})=\Lambda_{+}(B_{0}\oplus0+0\oplus C_{1}).\]
It follows that $(\alpha,\beta,\gamma')$ and $(\alpha,\beta,\gamma'')$
satisfy all the Horn and reverse Horn inequalities, and moreover $\gamma_{n}^{\prime\prime}>\gamma'_{n}$
for all $n$. This situation is not possible in the case of summable
sequences. Indeed, the trace identity would imply that $\sum_{n=1}^{\infty}(\gamma'_{n}-\gamma_{n}^{\prime\prime})=0.$

Note moreover that for any sequence $\gamma$ such that $\gamma'\le\gamma\le\gamma''$,
the triple $(\alpha,\beta,\gamma)$ satisfies all the Horn and reverse
Horn inequalities, and therefore there exist positive compact operators
$A,B,C$ such that $\Lambda_{+}(A)=\alpha,$ $\Lambda_{+}(B)=\beta$,
$\Lambda_{+}(C)=\gamma$, and $A=B+C$. A particular case is the sequence
\[
\gamma=(1,\frac{1}{4},\frac{1}{6},\dots,\frac{1}{2n},\dots).\]
Since $\alpha_{1}+\alpha_{2}=\beta_{1}+\gamma_{1}$, it follows that
$A,B,C$ have a common reducing space $\mathfrak{M}$ of dimension 2,
such that $\Lambda(A|\mathfrak{M})=(\alpha_{1},\alpha_{2}),$ $\Lambda(B|\mathfrak{M})=(\beta_{1},0)$,
and $\Lambda(C|\mathfrak{M})=(\gamma_{1},0)$. Thus the operators $B$
and $C$ must necessarily have the eigenvalue zero. Restricting the
operators $A,B$ and $C$ to $\mathfrak{M}^{\perp}$, we see that the
sequences $\widetilde{\alpha}=(1/(n+2))_{n=1}^{\infty}$, $\widetilde{\beta}=\widetilde{\gamma}=(1/2(n+1))_{n=1}^{\infty}$
must satisfy all the Horn and extended reverse Horn inequalities.
We will return to these sequences a little later.

It is an amusing exercise to show that, in the case of finite traces,
the trace identity does follow from the direct and reverse Horn inequalities.

Let us note that the triple $(\alpha=2\beta,\beta,\gamma=(1+\varepsilon)\beta)$,
where $\beta=(1/2n)_{n=1}^{\infty}$ does not satisfy all the reverse
Horn inequalities if $\varepsilon>0$. Indeed, among these inequalities
is \[
\sum_{i=1}^{2N}\alpha_{i}\ge\sum_{j=1}^{N}\beta_{j}+\sum_{k=1}^{N}\gamma_{k}.\]
 The three sums are asymptotic to $\log N$, $(1/2)\log N$, and $((1+\varepsilon)/2)\log N$.
This can also be seen directly from the additivity of the Dixmier
trace applied to the compact operators that would have these sequences
as eigenvalues.

As pointed out in the introduction, a different characterization of
the triples $(\Lambda_{+}(B+C),\Lambda_{+}(B),\Lambda_{+}(C))$ was
given in \cite{BLS} for positive compact operators $B,C$. This involves
a continuous version of the Littlewood-Richardson rule. It was pointed
out to us by Cristian Lenart \cite{lenart} that it might be possible
to reformulate this rule in terms of honeycombs or hives \cite{KT1,KT2,buch-sat},
and we take the opportunity to do this. For our purposes, a hive will
simply be a concave function $f:[0,\infty)^{2}\to\mathbb{R}$, whose
restriction to the triangles\[
T_{ij}=\text{co}\{(i-1,j),(i-1,j-1),(i,j-1)\},\quad S_{ij}=\text{co}\{(i-1,j),(i,j),(i,j-1)\}\]
is affine for all $i,j\in\mathbb{N}$; here we use co$(S)$ to denote
the convex hull of $S$. Such a function is determined, up to an additive
constant, by the points $(x_{ij},y_{ij},z_{ij})\in\mathbb{R}^{3}$
defined by\[
x_{ij}=f(i,j-1)-f(i-1,j-1),\quad y_{ij}=f(i-1,j-1)-f(i-1,j),\quad z_{ij}=f(i-1,j)-f(i,j-1)\]
for $i,j\in\mathbb{N}$. These points form the vertices of an infinite
honeycomb in the plane $x+y+z=0$; see \cite{KT1,KT2} for a discussion
of finite honeycombs. Decreasing sequences $\alpha,\beta,$ and $\gamma$
converging to zero and such that $\alpha\ge\beta$ satisfy the continuous
Littlewood-Richardson rule if and only if there exists a hive $f$
such that \[
x_{i0}=\alpha_{i},\quad y_{0j}=-\beta_{j},\quad\lim_{j\to\infty}z_{ij}=-\gamma_{i}\quad\text{for }i,j\in\mathbb{N}.\]
Generally it is rather difficult to construct a hive $f$ realizing
this rule. We will construct a hive for the particular case of the
sequences $\widetilde{\alpha},\widetilde{\beta}$ and $\widetilde{\gamma}$
considered above. Note that the hive will be determined up to an additive
constant from the values of $x_{i0}$, $y_{0j}$ and $z_{ij}$. The
reader will verify without difficulty that the following values do
yield a hive:\[
z_{ij}=\frac{1}{2}\left[\frac{1}{i+j+1}-\frac{1}{i+1}\right],\quad i,j\in\mathbb{N}.\]

\end{document}